\documentclass[11pt,reqno]{amsart}
\usepackage{amssymb}
\usepackage[T1]{fontenc}
\usepackage[latin1]{inputenc} 
\usepackage{dsfont, mathrsfs}
\usepackage{a4wide}  
\usepackage{stmaryrd}
\usepackage{amsthm}
\usepackage{xspace}
\usepackage[hyphenbreaks]{breakurl}
\usepackage[hyphens]{url}

\usepackage[colorlinks, hyperindex]{hyperref}

\hypersetup{
    colorlinks=true,         
    linkcolor=blue,          
    citecolor=green,         
    filecolor=magenta,       
    urlcolor=cyan            
}

\theoremstyle{plain}
\newtheorem{theorem}{Theorem}[section]
\newtheorem{lemma}[theorem]{Lemma}

\theoremstyle{definition}
\newtheorem{definition}[theorem]{Definition}
\newtheorem{example}[theorem]{Example}

\newtheorem{question}[theorem]{Question}

\theoremstyle{remark}
\newtheorem{remark}[theorem]{Remark}

\newcommand{\Cc}{\mathbb{C}}

\newcommand{\Ee}{\mathbb{E}}
\newcommand{\Ff}{\mathbb{F}}

\newcommand{\Nn}{\mathbb{N}}
\newcommand{\Pp}{\mathbb{P}}

\newcommand{\Rr}{\mathbb{R}}

\newcommand{\Un}{\mathds{1}}

\newcommand{\Le}{\mathcal{L}}

\newcommand{\Pe}{\mathcal{P}}
\newcommand{\Qe}{\mathcal{Q}}

\newcommand{\Ns}{\mathscr{N}}
\newcommand{\Ps}{\mathscr{P}}
\newcommand{\Qs}{\mathscr{Q}}

\newcommand{\Us}{\mathscr{U}}

\newcommand{\Rg}{\mathfrak{R}}
\newcommand{\Sg}{\mathfrak{S}}

\newcommand{\eg}{\mathfrak{e}}

\newcommand{\ensemble}[1]{ \left\lbrace #1 \right\rbrace } 
\newcommand{\prth}[1]{\!\left( #1 \right) }
\newcommand{\crochet}[1]{\left[ #1 \right] }  
\newcommand{\intcrochet}[1]{\llbracket #1 \rrbracket} 
\newcommand{\abs}[1]{\left| #1 \right|}  
\newcommand{\norm}[1]{\left| \! \left| #1 \right| \! \right|}

\newcommand{\Esp}[1]{ \Ee \prth{ #1 } }  
\newcommand{\Prob}[1]{ \Pp \prth{ #1 } }

\def\longlongrightarrow{\hspace{+0.1ex} - \hspace{-1.1ex} - \hspace{-1.1ex} - \hspace{-1.1ex}\longrightarrow  } 

\newcommand{\tendvers}[2]{ \underset{#1 \rightarrow #2}{\longlongrightarrow} }  
\newcommand{\cvlaw}[2]{\stackrel{\Le}{\underset{#1 \, \rightarrow \, #2}{\longlongrightarrow}}}  
\newcommand{\cvmodp}[2]{\stackrel{\operatorname{mod-P}}{\underset{#1 \, \rightarrow \, #2}{\longlongrightarrow}}}

\newcommand{\bracket}[1]{\langle #1 \rangle}
\newcommand{\Unens}[1]{ \Un_{ \ensemble{#1} } }
\def\eqlaw{\stackrel{\Le}{=}}
\newcommand{\pe}[1]{\left[ #1 \right] }

\def\ddivise{ {|} }

\newcommand{ \ppee }{\operatorname{p}}

\def\ee{\textbf{e}}

\def\geq{\geqslant}
\def\leq{\leqslant}

\def\Re{\Rg \eg}

\let\oldforall\forall
\def\forall{\oldforall\,} 

\let\oldexists\exists
\def\exists{\oldexists\,}

\newcommand{\emailhref}[1]{ \email{\href{mailto:#1}{#1}} }

\begin{document}


\title[A model reproducing the mod-Poisson fluctuations in the Selberg-Sath\'e theorem \ \ \ ]{ A penalised model reproducing the mod-Poisson fluctuations in the Sath\'e-Selberg theorem}


\author[Y. Barhoumi-Andr\'eani]{Yacine Barhoumi-Andr\'eani}
\address{Department of Statistics, University of Warwick, Coventry CV4 7AL, U.K.}
\emailhref{y.barhoumi-andreani@warwick.ac.uk}



\date{\today}


\begin{abstract}
We construct a probabilistic model for the number of divisors of a random uniform integer that converges in the mod-Poisson sense to the same limiting function as its original counterpart, the one arising in the Sath\'e-Selberg theorem. This construction involves a conditioning and gives an alternative perspective to the usual paradigm of ``hybrid product'' models developed by Gonek, Hughes and Keating in the case of the Riemann Zeta function.
\end{abstract}
\maketitle

\section{Introduction}

The Erd\"os-Kac theorem in probabilistic number theory concerns the Gaussian fluctuations of the number of prime divisors of a random integer : if $ \Pe $ denotes the set of prime numbers, let $ \omega(N) $ be the number of prime divisors of $ N \in \Nn $ defined by
\begin{align*}
\omega(N) :=  \sum_{ p \in \Pe } \Unens{ p \ddivise N }
\end{align*}
and let $ U_n $ be a random variable uniformly distributed in $ \ensemble{1, \dots, n} $. The Erd\"os-Kac theorem writes (see \cite{ErdosKac, ErdosKac2})
\begin{align}\label{Thm:ErdosKac}
\sup_{x \in \Rr} \abs{ \Prob{ \frac{ \omega(U_n) - \log\log n }{ \sqrt{ \log\log n }  }  \leq x } - \int_{-\infty }^x e^{- u^2 / 2 } \frac{ du }{ \sqrt{2 \pi } } }    \tendvers{n}{ + \infty } 0
\end{align}

The key understanding of this theorem is the following : the random variables 
\begin{align*}
B_p^{(n)} := \Unens{ p \ddivise U_n }
\end{align*}
are $ \ensemble{0, 1} $-Bernoulli random variables that are weakly correlated, and their approximation by a sequence of independent random variables is accurate at the the level of this Central Limit Theorem (CLT). Concretely, one can perform the approximation
\begin{align*}
\omega(U_n) = \sum_{ p \in \Pe } \Unens{ p \ddivise U_n }  = \sum_{ p \in \Pe } B_p^{(n)} \underset{n \to + \infty}{\approx} \sum_{ p \in \Pe, p \leq n } B_p^{(\infty)}
\end{align*}
the $ B_p^{(\infty)} $'s being independent Bernoulli random variables such that
\begin{align*}
\Prob{ \! B_p^{(\infty)} = 1 } = \frac{1}{p} = 1 - \Prob{ \! B_p^{(\infty)} = 0 }
\end{align*}

To measure the accuracy of this last approximation, we introduce the \textit{independent model}
\begin{align*}
\Omega_n := \sum_{ p \in \Pe, p \leq n } B_p^{(\infty)}
\end{align*}

A \textit{model} of a random or a deterministic sequence is a random sequence that can be substituted to its orginal in a prescribed framework while still capturing its main properties, for instance a particular type of convergence.

At the order of renormalisation of the CLT given by \eqref{Thm:ErdosKac}, the independent model is accurate since one can write
\begin{align*}
\frac{ \omega(U_n) - \log\log n }{ \sqrt{ \log\log n }  } \approx \frac{ \Omega_n - \sum_{ p \in \Pe, p \leq n } \frac{1}{p} }{ \sqrt{ \sum_{ p \in \Pe, p \leq n } \frac{1}{p}}  } \cvlaw{n}{+ \infty}   \Ns(0, 1)
\end{align*}

Here, we have used the well-known estimate for the \textit{prime harmonic sum}
\begin{align}\label{Def:PrimeHarmonicSum}
H^{(\Pe)}_n :=  \sum_{ p \in \Pe, p \leq n } \frac{1}{p} = \log\log n + O(1)
\end{align}

This model is interesting to understand the Erd\"os-Kac CLT, but it hides a certain amount of information since at the second order of renormalisation the dependency of the $ B_p^{(n)} $'s re-appears : one has the following result due to Selberg \cite{SelbergSathe} improving a result of Sath\'e \cite{Sathe}
\begin{align}\label{Thm:SelbergSatheEstimate}
\Esp{ z^{  \omega(U_n) } } = e^{ (z-1) (\log\log n + \kappa)} \Phi_\omega\prth{ z } \prth{ 1 + O\prth{ (\log n)^{\Re(z - 2)} } }
\end{align}
where, for $ R > 0 $, the $ O $ is uniform for $ \abs{z} \leq R $, where $ \kappa $ is an absolute positive constant and 
\begin{align}\label{Thm:SplittingModPoisson}
\Phi_\omega(z) := \prod_{k \in \Nn^* } \prth{ 1 + \frac{z - 1}{k} } e^{- \frac{z-1}{k} } \, \prod_{p \in \Pe } \prth{ 1 + \frac{z - 1}{p} } e^{- \frac{z-1}{p} }
\end{align}

But one has
\begin{align}\label{Thm:SelbergSatheIndependentEstimate}
\Esp{ z^{  \Omega_n } } = e^{ (z-1) (\log\log n + \kappa')} \Phi_\Omega\prth{ z } \prth{ 1 + O\prth{ (\log n)^{\Re(z - 2)} } }
\end{align}
with $ \kappa' $ another absolute constant and 
\begin{align*}
\Phi_\Omega(z) :=  \prod_{p \in \Pe } \prth{ 1 + \frac{z - 1}{p} } e^{- \frac{z-1}{p} }
\end{align*}
leading to a corrective factor $ \Phi_C(z) :=  \prod_{k \in \Nn^* } \! \prth{ 1 + \frac{z - 1}{k} } e^{- \frac{z-1}{k} } $ such that $ \Phi_\omega(z) \!= \Phi_\Omega(z) \Phi_C(z) $. This factor is easily seen to be the limiting function of the random variable
\begin{align*}
C_n := \sum_{k = 1}^n B_k^{(\infty)}
\end{align*}
which is equal in distribution to the number of cycles $ C(\sigma_n) $ of a random uniform permutation $ \sigma_n \in \Sg_n $ (see \cite{ArratiaBarbourTavare, KowalskiNikeghbali2}), hence the name.

$ $

Let $ (X_n) $ be a sequence of random variables and let $ P_{\gamma_n } $ be a Poisson-distributed random variable of parameter $ \gamma_n \in \Rr_+ $. When there exists a continuous function $ \Phi : A \subseteq \Cc \to \Cc $ satisfying some technical conditions (see definition \ref{Def:ModPoissonCv}) such that the following convergence holds locally uniformly in $ z \in A $
\begin{align*}
\frac{ \Esp{ z^{X_n} } }{ \Esp{ z^{ P_{\gamma_n } } } } \tendvers{n}{ + \infty } \Phi(z)
\end{align*}
one says that $ (X_n, \gamma_n)_n $ converges \textit{in the mod-Poisson sense} to $ \Phi $. 

This particular type of convergence was introduced in \cite{KowalskiNikeghbali2} following a similar development in the Gaussian setting in \cite{JacodAl}. 
It is unusual in Probability theory where the Fourier-Laplace transform is not often renormalised ; it implies the usual CLT by a change of renormalisation. This is the mode of convergence underlying equations \eqref{Thm:SelbergSatheEstimate} and \eqref{Thm:SelbergSatheIndependentEstimate} since 
\begin{align*}
\Esp{ z^{ P_{\gamma_n } } } = e^{\gamma_n (z-1) }
\end{align*}

A natural question arises from the last computations : 

\begin{question}\label{Q:QuestionArticle}

How to refine the independent model $ \Omega_n $ to get a model that would reproduce the mod-Poisson fluctuations, i.e. a model that would converge in the mod-Poisson sense to the function $ \Phi_\omega $ ?

\end{question}

The creation of heuristic probabilistic models aiming at understanding the sequence of prime numbers originates in the work of Cramer \cite{CramerRandom} and has recently seen some spectacular developments with the work \cite{LemkeSound} that precises Cramer's original idea that prime numbers behave ``at random''. The approach in this article is conceptually identical to the work of Hardy and Littlewood (see \cite{HardyLittlewoodOnCramer} or \cite{SoundSmallGaps}) that refines the coarse Cramer's model to incorporate effects likely to explain the distribution of twin primes. This refinement is done by biasing the probabilities of the Cramer model ; an enlightening description of the biasing procedure is given in \cite{SoundSmallGaps}. The analogy stops nevertheless here : the Cramer model consists in heuristically replacing the sequence of primes by a random sequence and argue that they behave in a ``similar way''~; as \cite{SoundSmallGaps} remarks, such an approach ``must be taken with a liberal dose of salt''. Here, no such heuristic replacement is carried out (for a comparison of unrelated probabilistic model in number theory, see \cite{TaoBlogCramer}). The problem here tackled is to construct a probabilistic approximation of a ``true'' random variable, $ \omega(U_N) $, but the approximation has to be understood in a peculiar probabilistic sense (the mod-Poisson one in place of, say, a total variation approximation), which motivates the terminology of \textit{model} such as defined previously.

The question \ref{Q:QuestionArticle} is motivated by a similar development for the random variable $ \zeta(1/2 + i t U) $ where $ U $ is a random variable uniformly distributed on the interval $ \crochet{0, 1} $, $ t > 0 $ and $ \zeta $ is the Riemann Zeta function. This random variable satisfies a CLT due to Selberg (see e.g. \cite{Joyner}) in the same vein as the Erd\"os-Kac one for $ \omega(U_n) $, namely
\begin{align*}
\sup_{x \in \Rr} \abs{ \Prob{ \frac{\log\abs{ \zeta(1/2 + i t U) } }{ \sqrt{ \frac{1}{2} \log\log t } } \leq x } - \int_{-\infty }^x e^{- u^2 / 2 } \frac{ du }{ \sqrt{2 \pi } } }    \tendvers{t}{ + \infty } 0
\end{align*}

The question of computing the limiting function for the ``mod-Gaussian renormalisation'' given for all $ \lambda \in i \Rr $ by
\begin{align*}
\frac{ \Esp{ e^{ \lambda \log\abs{\zeta(1/2 + i t U)} } }  }{ e^{ \frac{\lambda^2}{2} \times \frac{1}{2} \log\log t } } \tendvers{t}{ + \infty }  \Phi_\zeta(\lambda) = \Phi_{\operatorname{Matrix}}(\lambda) \Phi_{\operatorname{Arithmetic}}(\lambda)
\end{align*}
is the celebrated Keating-Snaith's moments conjecture (see \cite{KeatingSnaith} for the definitions).

In order to understand this last convergence, the authors of \cite{GonekHughesKeating} construct a ``hybrid product'' model for $ \log\abs{\zeta(1/2 + i t U)} $ that converges in the mod-* sense to the limiting function $ \Phi_\zeta $. Answering question \ref{Q:QuestionArticle} for the ``toy model'' given by $ \omega(U_n) $ is thus of importance for it may give a hint to understand $ \log\abs{\zeta(1/2 + i t U)} $.

We refer to \cite{GonekHughesKeating} for the exact details of the $ \zeta $ model ; instead of describing it, let us find its equivalent for $ \omega(U_n) $. Since the limiting function occuring in \eqref{Thm:SelbergSatheIndependentEstimate} is a product $ \Phi_\omega = \Phi_\Omega \Phi_C $, the idea is to think of $ \omega(U_n) $ as being approximately an independent sum of random variables created by means of Bernoulli random variables, i.e.
\begin{align}\label{Eq:HybridSum}
\omega(U_n) \approx \omega_n := \sum_{k = 1}^A B_k + \sum_{k = 1}^{A'} B'_k
\end{align}
with $ \Prob{B_k = 1} = \frac{1}{\ppee_k} $, $ \Prob{B'_k = 1} = \frac{1}{ k} $ and where $ A, A' $ are choosen so that the mod-Poisson speed of convergence $ \gamma_n = \log \log n + \kappa $ of $ \omega(U_n) $ matches the speed of convergence of $ \omega_n $. We have set $ \Pe := \ensemble{ \ppee_k, k \geq 1 } $. As $ \gamma_n $ is asymptotically the mean of $ \omega(U_n) $, using the classical relation $ \sum_{1 \leq k \leq n} 1/k = \log n + O(1) $ one finds
\begin{align*}
\Esp{ \omega_n } = \sum_{k = 1}^A \frac{1}{\ppee_k} + \sum_{k = 1}^{A'} \frac{1}{k} = \log \log A + \log (A') + O(1)
\end{align*}
which amounts to
\begin{align*}
A' \log A = O(\log n) 
\end{align*}
the constant in the $ O $ being explicitely known. 

This intuitive model, despite its degree of freedom and its artificial character, has the advantage of being an acceptable mod-Poisson model for $ \omega(U_n) $ since it converges in the mod-Poisson sense to $ \Phi_\omega = \Phi_\Omega \Phi_C $. Nevertheless, one could ask for another reason why the limiting correction to the independence takes the form of an additive independent term, and why this correction is again constructed by means of independent random variables. One can argue that a natural modification of the initial model is more likely to be understood by a biasing \`a la Hardy-Littlewood \cite{HardyLittlewoodOnCramer, SoundSmallGaps} instead of a summation paradigm. This is the goal of this article. 

More precisely, we will answer question \ref{Q:QuestionArticle} by conditioning a random proportion of primes to be divisors with probability one in a slight modification of $ \Omega_n $, i.e. by modifying the probabilities of the Bernoulli sum in the same way Hardy and Littlewood modify the probabilities of the Cramer Bernoulli variables, namely

\vspace{+0.25cm}
\noindent \textbf{Theorem.} Set $ \gamma_n := \log\log n $. For $ \theta > 0 $, let $ B_k(\theta) $ be the Bernoulli random variable given by
\begin{align*}
\Prob{ B_k(\theta) = 1} =  \frac{ \theta }{\theta + k - 1 } = 1 - \Prob{ B_k(\theta) = 0}
\end{align*}

There exist a real sequence $ (v_n)_n $, a random integer $ C'_n $ and a sequence $ (I_\ell)_\ell $ of i.i.d. random integers independent of $ (B_k(v_n))_k $ and $ C'_n $ (all quantities explicitely defined in theorem \ref{Theorem:ErdosKacImprovement}) such that
\begin{align*}
\Omega''_n   :\eqlaw \bigg( \sum_{ k   } B_{\ppee_k }(v_n) \, \bigg\vert B_{\ppee_{I_1} }(v_n) = \cdots = B_{\ppee_{I_{C'_n} }}(v_n)  = 1 \bigg)
\end{align*}
satisfies
\begin{align*}
\frac{  \Esp{ x^{ \Omega''_n  } } }{ \Esp{ x^{ P_{\gamma_n}  } } } \tendvers{n}{ + \infty } \Phi_\omega(x)  
\end{align*}

\vspace{+0.25cm}

The explicit description of all involved parameters is the content of theorem \ref{Theorem:ErdosKacImprovement}. It will be proven using a probabilistic interpretation of mod-Poisson convergence developed in \cite{BarhoumiModStein}, which interprets it as a change of probability. In the context of discrete random variables, conditioning and biasing can be understood in the same framework, which deepends the analogy with the Hardy-Littlewood approach.

\section*{Notations}

We gather here some notations used throughout the paper. 

The set $ \ensemble{1, 2, \dots, n } $ will be denoted by $ \intcrochet{1, n} $. The set of prime numbers will be denoted by $ \Pe := \ensemble{ \ppee_k, k \geq 1 } $.

The distribution of a real random variable $ X $ in the probability space endowed with the measure $ \Pp $ will be denoted by $ \Pp_X $ : if $ A $ is a measurable set, $ \Pp_X(A) := \Prob{X \in A } $. 

If $X$ and $Y$ are two random variables having the same distribution, that is $ \Pp_X = \Pp_Y $, we will note $ X \eqlaw Y $ or equivalently $ X \sim \Pp_Y $. We will denote by $ \Ps(\gamma) $ the Poisson distribution of parameter $ \gamma > 0 $, by $ \Ns(0, 1) $ the standard Gaussian distribution and by $ \Us(A) $ the uniform distribution on the set $ A $.

For $ f \in L^1(\Pp_X) $, $ f \geq 0 $, the penalisation or bias of $ \Pp_X $ by $f$ is the probability measure $ \Pp_Y $ denoted by
\begin{align*}
\Pp_{ Y } := \frac{ f(X) }{ \Esp{ f(X) } } \bullet \Pp_{ X }
\end{align*}

This definition is equivalent to the following : for all $ g \in L^\infty(\Pp_X) $, 
\begin{align*}
\Esp{ g(Y) } = \frac{ \Esp{ f(X) g(X) }  }{ \Esp{ f(X) } } 
\end{align*}

Note that in a discrete setting, conditioning amounts to take $ f = \Un_A $ for a suitable set $A$.

A partition of an integer $ N $ is a sequence of integers $ \lambda = (\lambda_1, \dots, \lambda_k) $ where $ \lambda_1 \geq \lambda_2 \geq \dots $ and $ \sum_{i = 1}^k \lambda_i = N $. We define the length of such a partition by $ \ell(\lambda) := k $. The paintbox process (see \cite{PitmanStFlour}) is a random partition $ \lambda = (\lambda_1, \dots, \lambda_k) $ constructed in the following way : let $ (I_i)_i $ be i.i.d. integer valued random variables and define the (random) equivalence relation by 
\begin{align}\label{Def:Paintbox}
k \sim r \ \ \ \Longleftrightarrow \ \ \ I_k = I_r
\end{align}

The equivalence classes of this relation define a random partition $ \lambda $. In the case where $ I \sim \Us\prth{ \intcrochet{1, N } } $, this random partition is equal in law to the cycle structure of a random uniform permutation $ \sigma \in  \Sg_N $ and in particular, the total number of cycles of $ \sigma $ satisfies $ C(\sigma) = \ell(\lambda) $.

\section{Reminder on mod-Poisson convergence}

\subsection{Definition and examples}

Let $ P_\gamma \sim \Ps(\gamma) $ with $ \gamma > 0 $. Recall that
\begin{align*}
\Pp(P_\gamma = k) = e^{- \gamma } \frac{\gamma^k}{k!} 
\end{align*}
which is a statement equivalent to $ \Esp{e^{i u P_\gamma} } = \exp\prth{ \gamma ( e^{iu } - 1 ) } $ for all $ u \in \Rr $. We define the mod-Poisson convergence in the Laplace-Fourier setting by the following 

\begin{definition}\label{Def:ModPoissonCv} Let $ (Z_n)_n $ be a sequence of positive random variables and $ (\gamma_n)_n $ be a sequence of strictly positive real numbers. $ (Z_n)_n $ is said to converge in the mod-Poisson sense at speed $ (\gamma_n)_n $ if for all $ z \in \Cc $, the following convergence holds locally uniformly in $ z \in \Cc $
\begin{align*}
\frac{ \Esp{ z^{  Z_n} } }{ \Esp{ z^{   P_{\gamma_n} } } } \tendvers{n }{ + \infty } \Phi\prth{ z }
\end{align*}
where $ \Phi : \Cc \to \Cc $ is a continuous function satisfying $ \Phi(1) = 1 $, $ \Phi(\overline{z}) = \overline{\Phi(z) } $ and with $ P_{\gamma_n} \sim  \Ps(\gamma_n) $.

When such a convergence holds, we write it as
\begin{align*}
(Z_n, \gamma_n ) \cvmodp{n }{ + \infty } \Phi
\end{align*}
\end{definition}

\begin{remark} The limiting function $ \Phi $ is not unique, since it is defined up to multiplication by an exponential (see \cite{JacodAl}).
\end{remark}

\begin{remark} The original definition used in \cite{KowalskiNikeghbali2} is in the Fourier setting, i.e. for $ \abs{z} = 1 $. The advantage of this definition is that the Fourier transform of a random variable always exists. 
Other restrictions of the domain of convergence are possible. For instance, \cite{FerayMeliotNikeghbali} uses $ \ensemble{ -c < \Re < c } $ for a certain $ c > 0 $. In the Laplace case, one has, locally uniformly in $ x \in \Rr_+ $
\begin{align*}
\frac{ \Esp{ x^{  Z_n} } }{ \Esp{ x^{   P_{\gamma_n} } } } \tendvers{n }{ + \infty } \Phi\prth{ x }
\end{align*}
and in particular, $ \Phi(x) \geq 0 $ for all $ x \in \Rr_+ $.
\end{remark}

From now on, we restrict ourselves to the Laplace setting. Mod-Poisson convergence will always mean ``in the Laplace setting'' unless specified. In particular, the limiting mod-Poisson function $ \Phi $ will be defined on $ \Rr_+ $ and a quantity such as $ \norm{ \Phi }_\infty $ will be understood as $ \underset{x \in \Rr_+ }{\sup} \abs{ \Phi(x) } $.

$ $

The following example is fundamental to understand mod-Poisson convergence :
\begin{example}\label{Ex:LimitePoissonGenerale} Let $ (B_k)_k $ be a sequence of Bernoulli random variables such that 
\begin{align*}
p_k := \Prob{B_k = 1} = 1 - \Prob{B_k = 0}
\end{align*}
where $ (p_k)_k $ is a sequence of real numbers satisfying the conditions
\begin{align*}
& (i) \,   \sum_{k \geq 1} p_k = + \infty \\
& (ii)     \sum_{k \geq 1} p_k^2 < + \infty
\end{align*}

Let $ Z_n := \sum_{k = 1}^n B_k $ and $ \gamma_n := \sum_{k = 1}^n p_k $. Then,
\begin{align*}
(Z_n, \gamma_n ) \cvmodp{n }{ + \infty } \Phi
\end{align*}
where
\begin{align}\label{LimitePoissonGenerale}
\Phi(x) = \prod_{k \geq 1 }  (1 + p_k (x-1) ) e^{ - p_k (x - 1) }
\end{align}

Indeed, setting $ P_{\gamma_n } \sim \Ps(\gamma_n) $ one has, locally uniformly in $ x $
\begin{align*}
\frac{ \Esp{ x^{Z_n} } }{ \Esp{ x^{ P_{\gamma_n } } } } & =  \frac{\prod_{k = 1}^n \Esp{ x^{B_k} }}{e^{ \gamma_n (x-1) }} = \prod_{k = 1}^n \prth{ 1 + p_k(x - 1) } e^{- p_k (x-1) }  \\
              & \tendvers{n }{ + \infty } \prod_{k \geq 1 }  (1 + p_k (x-1) ) e^{ - p_k (x - 1) }
\end{align*}
since $ \sum_k p_k^2 < \infty $ and $ (1 + p_k (x-1) ) e^{ - p_k (x - 1) } = \exp\prth{ - p_k^2 (x-1)^2 / 2 + o(p_k^2)  } $.
\end{example}

One can see that equation \eqref{Thm:SelbergSatheIndependentEstimate} is a particular case of this last theorem with $ p_k = 1/ \ppee_k $ where $ \Pe := \ensemble{ \ppee_k, k \geq 1 } $. As pointed out in the introduction, this is also the case of equation \eqref{Thm:SelbergSatheEstimate} using the ``hybrid sum'' $ \omega_n $ of \eqref{Eq:HybridSum} that incorporates the corrective term
\begin{align*}
\Phi_C(x) := \prod_{k \geq 1}  \prth{ 1 + \frac{x-1}{k} } e^{- \frac{x-1}{k} }  
\end{align*}

This term is the limiting mod-Poisson function of the random variable $ C_n := \sum_{k = 1}^n B_k $ where $ (B_k)_k $ is the last sequence of Bernoulli random variables with $ p_k = 1/k $, and with speed 
\begin{align*}
H_n := \sum_{k = 1}^n \frac{1}{k} = \log n + \gamma + o(1)
\end{align*}
where $ \gamma $ is the Euler-Mascheroni constant.

\begin{remark}\label{RemarqueCyclesPermutation}
Such a random variable has also the distribution of the total number of cycles of a random permutation selected according to the uniform distribution $ \Pp_n $ defined by $ \Pp_n(\sigma) = 1/ n! $ for all $ \sigma \in \Sg_n $ (see e.g. \cite{ArratiaBarbourTavare}). Using the formula (see e.g. \cite{WhittakerWatson} 12.11)
\begin{align*}
\frac{1}{ \Gamma(z ) } =  e^{ (z-1) \gamma } \prod_{k \geq 1} \prth{ 1 + \frac{z - 1}{k} } e^{ - \frac{z - 1}{k} }
\end{align*}
$ \Phi_C(z) $ can be replaced by $ 1/\Gamma(z) $ when the speed $ H_n $ is replaced by $ H_n - \gamma $.
\end{remark}


\subsection{A probabilistic interpretation of mod-Poisson convergence}

We recall the following theorem from \cite{BarhoumiModStein} :

\begin{theorem} 

Let $ \Phi $ be a bounded function on $ \Rr_+ $ and $ \gamma > 0 $. Define the distribution $ \Qs\prth{ \Phi, \gamma } $~by
\begin{align*}
Q_\gamma \sim  \Qs\prth{ \Phi, \gamma }  \ \  \Longleftrightarrow  \ \  \Pp_{ Q_\gamma } := \frac{ \Phi\prth{ \frac{P_\gamma}{\gamma} }  }{ \Esp{ \Phi\prth{ \frac{P_\gamma}{\gamma} } } }  \bullet \Pp_{ P_\gamma }  
\end{align*}
where $ P_\gamma \sim \Ps(\gamma) $.

Then, if $ \Qe_{\gamma_n}(\Phi) \sim  \Qs\prth{ \Phi, \gamma_n } $, we have
\begin{align*}
(\Qe_{\gamma_n}(\Phi), \gamma_n ) \cvmodp{n }{ + \infty } \Phi
\end{align*}
\end{theorem}

For the reader's convenience, we remind the proof of this result. 

\begin{proof} Recall the change of probability, for $ x, \gamma > 0 $
\begin{align}\label{PoissonExpBiais}
\frac{ x^{ P_{\gamma} } }{ \Esp{ x^{ P_{\gamma} } } }  \bullet \Pp_{  P_{\gamma}   } = \Pp_{ P_{x \gamma }  }
\end{align}
easily seen writing, for all $ \theta \in \Rr $
\begin{align*}
\frac{  \Esp{ x^{  P_{\gamma } } \, e^{ i \theta P_\gamma    }   } }{  \Esp{ x^{ P_\gamma  } } } = \frac{ e^{\gamma \prth{ x e^{ i \theta } - 1 }    }  }{ e^{\gamma \prth{ x - 1 }    } } = e^{ \gamma x \prth{ e^{ i \theta  } - 1 } }  =  \Esp{ e^{ i \theta P_{x\gamma}   } }  
\end{align*}

Then, we have
\begin{align*}
\frac{ \Esp{x^{ \Qe_{\gamma_n}(\Phi) } } }{\Esp{ x^{ P_{\gamma_n} } } } =  \frac{ \Esp{  \Phi\prth{ \frac{ P_{\gamma_n} }{\gamma_n} } x^{ P_{\gamma_n} } } }{  \Esp{ x^{ P_{\gamma_n} } } \Esp{  \Phi\prth{ \frac{ P_{\gamma_n} }{\gamma_n} } } } =  \frac{   \Esp{  \frac{ \vphantom{ \big ( } x^{ P_{\gamma_n} } }{  \Esp{  \vphantom{ \big ( } x^{ P_{\gamma_n} } } }   \Phi\prth{ \frac{ P_{\gamma_n} }{\gamma_n} }    }  }{ \Esp{  \Phi\prth{ \frac{ P_{\gamma_n} }{\gamma_n} } } }  = \frac{ \Esp{  \Phi\prth{ \frac{ P_{x\gamma_n} }{\gamma_n}  } } }{ \Esp{  \Phi\prth{ \frac{ P_{\gamma_n} }{\gamma_n} } }}
\end{align*}

By continuity and boundedness of $ \Phi $, and using dominated convergence and the fact that 
\begin{align*}
\frac{ P_{x \gamma_n } }{\gamma_n }  \cvlaw{n}{ + \infty } x 
\end{align*}
one gets, locally uniformly in $ x \in \Rr_+ $ (and in particular for $ x = 1 $)
\begin{align*}
\Esp{  \Phi\prth{ \frac{ P_{x\gamma_n} }{\gamma_n}  } } \tendvers{n}{ + \infty } \Phi(x)
\end{align*}

As $ \Phi(1) = 1 $, one finally gets the result.
\end{proof}


\begin{example} Continuing example \ref{Ex:LimitePoissonGenerale}, we see that in the case of a function given by \eqref{LimitePoissonGenerale}, i.e. 
\begin{align*}
\Phi(x) = \prod_{k \geq 1 }  (1 + p_k (x-1) ) e^{ - p_k (x - 1) }
\end{align*}
one has, for all $ x \in \Rr_+ $
\begin{align*}
0 \leq \Phi(x) \leq 1
\end{align*}
and the last theorem applies. The positivity of $ \Phi $ on $ \Rr_+ $ is obvious, and as $ 1 + y \leq e^y $ for all $ y \in \Rr $, setting $ y = p_k (x-1) $ one has $ (1 + p_k (x-1) ) e^{ - p_k (x - 1) } \leq 1 $ which gives the upper bound.
\end{example}

A probabilistic interpretation of mod-Poisson convergence follows from this last theorem~: if $ (Z_n)_n $ is a sequence of random variables converging in the mod-Poisson sense at speed $ (\gamma_n)_n $ to $ \Phi $, one may think of the distribution of $ Z_n $ as close to the distribution of $ \Qe_{\gamma_n}(\Phi) $. The limiting function $ \Phi $, once correctly scaled, would thus be a particular correction to the Poisson distribution that would allow a refined speed of convergence in the CLT, and mod-Poisson convergence could thus be understood as a certain second-order convergence in distribution. This is the case in the mod-Gaussian setting (see \cite{BarhoumiModStein}), but also in the mod-Poisson setting since \cite[II.6 (20)]{Tenenbaum} gives
\begin{align*}
\Prob{ \omega(U_n) = k } = \Prob{ P_{\log \log n} = k } \prth{   \Phi_\omega \prth{ \frac{k}{\log\log n} } + O\prth{ \frac{k}{ (\log \log n)^2 } }  }
\end{align*}
uniformly in $ n \geq 3 $ and $ k \in \intcrochet{ 1, (2 - \delta)\log\log n } $ for all $ \delta > 0 $. Moreover, using \ref{Lemma:OrderPhiC}, one has $ \norm{\Phi_\omega' }_\infty < \infty $, which implies using the Gaussian CLT for $ P_\gamma $
\begin{align*}
\Esp{ \abs{ \Phi_\omega \prth{ \frac{ P_{\gamma} }{\gamma} } -1 } } = \Esp{ \abs{ \Phi_\omega \prth{ \frac{ P_{\gamma} }{\gamma} } - \Phi_\omega(1) } } \leq  \norm{\Phi_\omega' }_\infty \Esp{ \abs{ \frac{ P_{\gamma} }{\gamma} - 1 } }  = O\prth{ \frac{ 1 }{\sqrt{\gamma } } }
\end{align*}

This last result can hence be transformed into
\begin{align*}
\Prob{ \omega(U_n) = k } = \Prob{ P_{\log \log n} = k } \prth{   \frac{ \Phi_\omega \prth{ \frac{k}{\log\log n} } }{  \Esp{ \Phi_\omega \prth{ \frac{ P_{\log \log n} }{\log\log n} } } } + O\prth{ \frac{ \Esp{ \Phi_\omega \prth{ \frac{  P_{\log \log n}  }{\log\log n} }} }{ \sqrt{\log \log n } } } + O\prth{ \frac{k}{ (\log \log n)^2 } }  }
\end{align*}
that is
\begin{align*}
\Prob{ \omega(U_n) = k } 
 						& = \Prob{ \Qe_{\log\log n}(\Phi_\omega) = k }  + O\prth{ \frac{ (\log \log n)^{k - 1/2 } }{ k! \ \log n}  } 
\end{align*}

\section{A model that converges in the mod-Poisson sense}

In order to construct a probabilistic model for $ \omega(U_n) $ that converges in the mod-Poisson sense to $ \Phi_\omega $, we remind some classical probabilistic biases.

\subsection{Classical biases and changes of probability}

A fundamental operation in probability theory is the change of probability by means of a weight on the initial probability measure. This weight is called \textit{bias} or \textit{penalisation} and we will use undifferently both terminology. 

\begin{definition}[Bias/penalisation of measure] Let $ X $ be a real random variable in the probability space endowed with the measure $ \Pp $ and denote by $ \Pp_X $ its law. For $ f \in L^1(\Pp_X) $, $ f \geq 0 $, the penalisation (or bias) of $ \Pp_X $ by $f$ is the probability measure $ \Pp_Y $ denoted by
\begin{align*}
\Pp_{ Y } := \frac{ f(X) }{ \Esp{ f(X) } } \bullet \Pp_{ X }
\end{align*}
\end{definition}

Classical bias in probability theory allow to understand ``pathwise transformations'' induced by such a transformation.

\begin{example} The most classical change of probability concerns the passage from $ \Ns(0, 1) $ to $ \Ns(\mu, 1) \eqlaw \mu + \Ns(0, 1) $. Indeed, if $ X \sim \Ns(0, 1) $, one easily checks that
\begin{align*}
\Pp_{ X + \mu } = \frac{ e^{ \mu X   } }{ \Esp{ e^{ \mu X  } } } \bullet \Pp_{ X } = e^{ \mu X - \mu^2 / 2 } \bullet \Pp_{ X }
\end{align*}

Hence, in the Gaussian setting, an exponential bias is equivalent to a translation of the canonical evaluation. Note that the Poisson counterpart of this exponential bias was given in equation \eqref{PoissonExpBiais}. 
\end{example}

A classical transform in probability theory is made using $ f : x \mapsto x $ when the random variable is positive.

\begin{definition}[Size-bias transform] Let $ X \geq 0 $ be a random variable with expectation $ \mu := \Esp{X} < \infty $. A random variable $ X^{(s)} $ is said to be a \textit{size-bias transform} of $ X $ if, for all real functions $f$ such that $ \Esp{ \abs{ X f(X) } } < \infty $
\begin{align*}
\Esp{  X f(X)  } = \mu \Esp{  f\prth{ X^{(s)} }  } 
\end{align*}

An equivalent definition is thus 
\begin{align}\label{SizeBiais}
\Pp_{ X^{(s)} } := \frac{X}{\Esp{  X  } } \bullet \Pp_X 
\end{align}
\end{definition}


\begin{example} A classical change of measure for a random walk is given by its size-bias coupling, i.e. given $ (X_k)_k $ a sequence of i.i.d. positive random variables of expectation $ \Esp{X_k} := 1 $ defined on the same probability space, the random walk $ (S_n)_n $ of increments $ (X_k)_k $ is given by
\begin{align*}
S_n := \sum_{k = 1}^n X_k 
\end{align*}

The size-bias transform of $ S_n $ is the random variable $ S_n^{(s)} $ whose law is given by 
\begin{align*}
\Pp_{ S_n^{(s)} } := \frac{S_n }{ n  } \bullet \Pp_{S_n }
\end{align*}
\end{example}

A pathwise construction of such a random variable is implied by the following

\begin{lemma}[Size-bias coupling of an independent sum]\label{CouplageBiaisTaille} Let $ (Y_k)_k $ be a sequence of independent positive integrable random variables, independent of $ (X_k)_k $ and having the same distribution as $ (X_k)_k $ and let $ I \in \intcrochet{1, n} $ be a random index independent of $ (X_k)_k $ and $ (Y_k)_k $ of law given by
\begin{align*}
\Prob{ I = k } = \frac{ \Esp{X_k } }{ \sum_{ \ell = 1}^n \Esp{X_\ell } }
\end{align*}

Then, 
\begin{align*}
S_n^{(s)} \eqlaw S_n - X_I + Y^{(s)}_I 
\end{align*}
and in particular, if $ (Y_k)_k $ is defined on the same probability space as $ (X_k)_k $, one has a natural coupling $ \prth{S_n, S_n^{(s)} } $.
\end{lemma}

For the reader's convenience, we recall the proof of this lemma (see also e.g. \cite[pp 78-79]{ArratiaBarbourTavare}).


\begin{proof} Let $ f $ be a bounded measurable function and $ S_n^{ \bracket{ - k } } := \sum_{\ell \neq k} X_\ell $. Then, by independence,
\begin{align*}
\Esp{ f\prth{ S_n^{(s)} } } & := \frac{1}{ \Esp{ S_n } } \Esp{ S_n  f( S_n ) } = \frac{1}{ \Esp{ S_n } } \sum_{k = 1}^n \Esp{ X_k  f( S_n ) } \\
							& = \frac{1}{ \Esp{ S_n } } \sum_{k = 1}^n \Esp{ X_k  f\prth{ S_n^{ \bracket{ - k } } + X_k  } } \\
							& = \frac{1}{ \Esp{ S_n } } \sum_{k = 1}^n \Esp{ X_k } \Esp{  f\prth{ S_n^{ \bracket{ - k } } + Y^{(s)}_k  } } \\
							& = \Esp{  f\prth{ S_n^{ \bracket{ - I } } + Y^{(s)}_I  } } = \Esp{  f\prth{ S_n - X_I + Y^{(s)}_I  } }
\end{align*}
\end{proof}


A last type of useful bias concerns the exponential bias of a sum of independent terms, and in particular Bernoulli random variables. 
\begin{lemma}[Exponential bias of an independent sum]\label{CouplageBiaisExp} Let $ (Y_k)_k $ be a sequence of independent random variables. Define, for a certain $ x > 0 $,
\begin{align*}
S_n  & := \sum_{k = 1}^n Y_k \\
\Pp_{S_n(x)} & =  \frac{ x^{S_n} }{ \Esp{x^{S_n} } } \bullet \Pp_{ S_n }
\end{align*}
and suppose that $ \Esp{x^{Y_k} } < \infty $ for all $ k \geq 1, \ x \in \Rr_+ $. Then, 
\begin{align*}
S_n(x) \eqlaw \sum_{k = 1}^n Y_k(x)
\end{align*}
with 
\begin{align*}
\Pp_{Y_k(x)} = \frac{ x^{Y_k} }{ \Esp{x^{Y_k} } } \bullet \Pp_{ Y_k }
\end{align*}

In particular, if $ (B_k)_k $ is a sequence of independent $ \ensemble{0, 1} $-Bernoulli random variables, each of probability $ p_k $ to be equal to $ 1 $, then, 
\begin{align*}
\Prob{ B_k(x) = 1 } = p_k(x) := \frac{x p_k}{x p_k + 1 - p_k } = 1 - \Prob{ B_k(x) = 0 }
\end{align*}
\end{lemma}

For self-completeness, we recall the proof of this result.


\begin{proof} 
Let $ y > 0 $. Then, by independence,
\begin{align*}
\Esp{ y^{ S_n(x) } } := \frac{  \Esp{ (xy)^{ S_n } }  }{\Esp{ x^{ S_n } }} = \prod_{k = 1}^n \frac{  \Esp{ (xy)^{ Y_k } }  }{\Esp{ x^{ Y_k } }} =: \prod_{k = 1}^n \Esp{ y^{ Y_k(x) } }
\end{align*}

In particular, 
\begin{align*}
\frac{  \Esp{ (xy)^{ B_k } }  }{\Esp{ x^{ B_k } }} = \frac{ p_k xy + 1 - p_k }{p_k x  + 1 - p_k } = 1 + p_k(x) (y-1) = \Esp{ y^{ B_k(x) } } 
\end{align*}
\end{proof}


\subsection{Estimates on the limiting function} 

Proving theorem \ref{Theorem:ErdosKacImprovement} requires some preparation.

\begin{lemma}\label{Lemma:OrderPhiC} Define 
\begin{align*}
\phi_k(x) := \prod_{\ell = 1}^k \prth{ 1 + \frac{x- 1}{ \ell } } e^{ -\frac{x-1}{ \ell } }
\end{align*}

Then, when $ k \to + \infty $
\begin{align*}
\sup_{\Rr_+} \abs{ \phi'_k} = O(1)
\end{align*}
\end{lemma}


\begin{proof}

We have for all $ x \geq 0 $ and all $ \ell \geq 1 $
\begin{align*}
\prth{ 1 + \frac{x - 1}{\ell	} } e^{ - \frac{x - 1}{\ell	} } \leq 1
\end{align*}
hence 
\begin{align*}
\phi_k(x) =  x e^{1 - x} \prod_{\ell = 2}^k \prth{ 1 + \frac{x - 1}{\ell	} } e^{ - \frac{x - 1}{\ell	} }  \leq x e^{1 - x} 
\end{align*}

Moreover, 
\begin{align*}
\phi'_k(x) = \phi_k(x) (1- x) \prth{ \sum_{ \ell = 1 }^k  \frac{1}{\ell( \ell - 1 + x) } }
\end{align*}
from what we deduce that for all $ x \geq 0 $
\begin{align*}
\abs{ \phi'_k(x) }  & = \phi_k(x) \abs{1 - x} \prth{ \frac{1}{x } + \sum_{ \ell = 2 }^k \frac{1}{\ell(\ell - 1 + x) } } \\
                    & \leq e^{1 - x} \, x \abs{1 - x} \prth{  \frac{1}{x } + \sum_{ \ell = 2 }^k \frac{1}{\ell(\ell - 1 ) } } = e^{1 - x} \, x \abs{1 - x} \prth{  \frac{1}{x } + 1 - \frac{1}{k} } \\
                    & \leq e^{1-x} \, x \abs{1 - x} \prth{  \frac{1}{x } + 1 } =  e^{1 - x} \abs{ 1 - x^2 } \\
                    & \leq e
\end{align*}
after a study of $ x \mapsto e^{1 - x} \abs{ 1 - x^2 } $.
\end{proof}


\begin{lemma}\label{Lemma:OrderDeltaPhiC} Set $ \phi := \Phi_C $. Then, when $ k \to + \infty $
\begin{align*}
\sup_{\Rr_+} \abs{ \phi_k -  \phi } & = O \prth{ \frac{1}{ \sqrt{k} } }  
\end{align*}
\end{lemma}


\begin{proof}
We split the proof into different cases : $ x \in \crochet{0, 1} $ and $ x \in [1, + \infty [ $.

$ $ 

\noindent $ \bullet $ \underline{ $ x \in \crochet{0, 1} $ :} We have $ \phi_k - \phi = \phi_k (1 - \psi_k)  $ where 
\begin{align*}
\psi_k(x) := \prod_{\ell \geq k + 1} \prth{ 1 - \frac{1 - x}{\ell	} } e^{  \frac{1 - x}{\ell	} }
\end{align*}
hence $ \phi_k(x) - \phi(x) \leq 1 - \psi_k(x) $ for all $ x \in \crochet{0, 1} $, since $ \phi_k(x) \leq 1 $ and $  1 - \psi_k(x) \geq 0 $. But
\begin{align*}
- \log \psi_k(x)  & = - \sum_{ \ell \geq k + 1 } \crochet{ \frac{1 - x}{\ell } + \log\prth{ 1 - \frac{1-x}{\ell} }  } \\
                  & = - \sum_{ \ell \geq k + 1 } \int_0^{1 - x} \prth{ \frac{1}{\ell } - \frac{1}{ \ell ( 1 - u / \ell ) } }\, du \\
                  & = \sum_{ \ell \geq k + 1 } \frac{1}{\ell^2} \int_0^{1 - x}  \frac{u}{ 1 - u/\ell} \, du \\
                  & =  \sum_{ \ell \geq k + 1 } \frac{1}{\ell^2}  \int_0^1  \frac{u}{ 1 - u / 2} \, du \ \ \ \forall k \geq 1 
\end{align*}

For all $ t \in \crochet{ \ell, \ell + 1} $, we have $ (\ell + 1)^{-2} \leq t^{-2}  \leq \ell^{-2} $ ; integrating this inequality on $  \crochet{ \ell, \ell + 1} $ 
and summing on $ \ell \geq k $ gives
\begin{align*}
\sum_{\ell \geq k} \frac{1}{ (\ell + 1)^2 } \leq \int_\ell^{ + \infty } \frac{dt}{t^2} \leq \sum_{\ell \geq k} \frac{1}{\ell^2}
\end{align*}
%
%
%

We hence have
\begin{align*}
\sum_{\ell \geq k + 1} \frac{1}{ \ell^2 } \leq \frac{1}{k} 
\end{align*}
which implies that for all $ k \geq 1 $
\begin{align*}
- \log \psi_k(x)  \leq \frac{C}{k} \ \ \ \ \ \mbox{with } \ \ C :=   \int_0^1  \frac{u}{ 1 - u / 2} \, du <  \infty 
\end{align*}

Last, 
\begin{align*}
1 - \psi_k(x) = 1 - e^{- (-\log \psi_k(x) ) } \leq 1 - e^{ -C / k  } \leq \frac{C}{k}
\end{align*}
which gives a stronger result, i.e. $ \sup_{x \in \crochet{0, 1} }\abs{ \phi_k(x) - \phi(x) } = O(1/k) $.

$ $

\noindent $ \bullet $ \underline{ $ x \in \crochet{1, + \infty} $ :} Let $ \ee, \ee' $ be two independent exponentially-distributed random variables, i.e. $ \Prob{\ee \geq x } = e^{-x} $ for all $ x \geq 0 $, and let $ Z $ be such that
\begin{align*}
\Prob{ Z \geq x } = (1 + x)e^{-x}
\end{align*}
namely $ Z \eqlaw \ee + \ee' $ since for all $ \lambda \geq 0 $, $ \Esp{ e^{-\lambda Z} } = (1 + \lambda)^{-2} =  \Esp{ e^{ -\lambda(\ee + \ee') } } $.
%
%
%

Let $ (Z_\ell)_\ell $ be a sequence of i.i.d. random variables distributed as the sum of two independent exponential random variables. Then, for all $ y \geq 0 $
\begin{align*}
\psi_k(y + 1) = \prod_{\ell \geq k + 1 } \prth{ 1 + \frac{y }{\ell} } e^{-\frac{y}{\ell} } = \prod_{\ell \geq k + 1} \Prob{ Z_\ell \geq \frac{y}{\ell} } = \Prob{ \min_{\ell \geq k + 1} \!\! \ensemble{ \ell Z_\ell } \geq y }
\end{align*}

This implies that
\begin{align*}
1 - \psi_k(y + 1) =  \Prob{ y \geq \min_{\ell \geq k + 1} \!\! \ensemble{ \ell Z_\ell }  } = \Prob{ y \max_{\ell \geq k + 1} \!\! \ensemble{ \frac{1}{\ell Z_\ell} } \geq 1 } \leq y \, \Esp{  \max_{\ell \geq k + 1} \!\! \ensemble{ \frac{1}{\ell Z_\ell} }  }
\end{align*}

As $ y \, \phi_k(y + 1) \leq y(y + 1) e^{-y } \leq 2 + \sqrt{5} $, it is enough to show that 
\begin{eqnarray*}
\Esp{  \max_{\ell \geq k + 1} \!\! \ensemble{ \frac{1}{\ell Z_\ell} }  } = O\prth{ \frac{1}{k^{1/2  } } } 
\end{eqnarray*}

For all $ \eta > 0 $, write
\begin{align*}
\Esp{  \max_{\ell \geq k + 1} \!\! \ensemble{ \frac{1}{\ell Z_\ell} }  }   & =    \int_{\Rr_ + } \Prob{ \max_{\ell \geq k + 1} \!\! \ensemble{ \frac{1}{\ell  Z_\ell} } \geq t } dt = \int_{\Rr_ + } 1 - \Prob{ \max_{\ell \geq k + 1} \!\! \ensemble{ \frac{1}{\ell  Z_\ell} } \leq t } dt \\
                  & =  \int_{\Rr_ + } \prth{ 1 - \prod_{\ell \geq k + 1} \Prob{  \frac{1}{\ell  Z_\ell}   \leq t } } dt = \int_{\Rr_ + } \prth{ 1 - \prod_{\ell \geq k + 1} \Prob{  \ell  Z \geq u } } \frac{du}{u^2} \\
                  & \leq \int_0^\eta \prth{ 1 - \prod_{\ell \geq k + 1} \Prob{ \ell Z \geq u } } \frac{du}{u^2} + \int_\eta^{ +\infty } \frac{du}{u^2}\\
                  & =:  \int_0^\eta \prth{ 1 - e^{- u_k(u)} } \frac{du}{u^2} + \frac{1}{\eta }
\end{align*}
where
\begin{align*}
u_k(u)   & := - \sum_{ \ell \geq k + 1 }  \log\Prob{ Z \geq t \ell^{- 1} } \\
                        & =  - \sum_{ \ell \geq k + 1 } - \frac{u }{\ell  } + \log\prth{ 1 + \frac{u }{\ell } } \\
                        & =    \sum_{ \ell \geq k + 1 } \frac{1}{\ell^2 } \int_0^u \frac{v}{1 +  \ell^{-1} v } dv \\
                        & \leq    \frac{1}{  k  } \int_0^u v dv  = \frac{u^2}{2 k }
\end{align*}

Hence
\begin{align*}
\int_0^\eta \prth{ 1 - e^{- u_k(u)} } \frac{du}{u^2} \leq \frac{\eta}{2 k }
\end{align*}
which finally gives
\begin{align*}
\Esp{  \max_{\ell \geq k + 1} \!\! \ensemble{ \frac{1}{\ell Z_\ell} }  }   \leq    \prth{\frac{\eta}{2 k } + \frac{1}{\eta}  }
\end{align*}

Optimising in $ \eta $ amounts to choose $ \eta = \sqrt{ 2 k } $ and gives
\begin{align*}
\Esp{  \max_{\ell \geq k + 1} \!\! \ensemble{ \frac{1}{\ell Z_\ell} }  }   \leq   \frac{2}{\sqrt{2 k  } }  
\end{align*}
which is the desired result.
\end{proof}

\subsection{Construction of the model}\label{PartieConstruction} 

In what follows, we use the conventions of the last lemmas. In particular, we will note $ Y(x) $ the $ x $-exponential bias of the random variable $ Y $ as in lemma \ref{CouplageBiaisExp}. We recall that $ H^{(\Pe)}_n \underset{n \to +\infty}{\sim} \log \log n $ was defined in equation \eqref{Def:PrimeHarmonicSum} and that $\pi(n) := \sum_{p \in \Pe} \Unens{p \leq n}  \underset{n \to +\infty}{\sim} \frac{ n }{\log n } $ by the prime number theorem. We moreover define
\begin{align*}
k_n   & :=  \pe{ H^{(\Pe)}_n } \underset{n \to +\infty}{\sim}  \log \log n  \\
v_n   & :=  \exp\prth{ -\frac{ H_{k_n} }{ k_n } } = \exp\prth{ - \frac{ \log\log\log n + O(1) }{  \log\log n} } 
\end{align*}
\begin{theorem}[A mod-Poisson improvement of the Erdös-Kac model]\label{Theorem:ErdosKacImprovement} Let $ (B_k)_k $ and $ (B'_k)_k $ be two independent sequences of independent $ \ensemble{0, 1} $-Bernoulli random variables with probabilities
\begin{align*}
\Prob{ B_k = 1} = \Prob{ B'_k = 1} = \frac{1}{k}
\end{align*}

For $ \theta > 0 $, let $ B_k(\theta) $ be the $ \theta $-exponential bias of $ B_k = B_k(1) $ given by
\begin{align*}
\Prob{ B_k(\theta) = 1} =  \frac{ \theta }{\theta + k - 1 } = 1 - \Prob{ B_k(\theta) = 0}
\end{align*}
or equivalently
\begin{align*}
\Pp_{ B_k(\theta) } :=  \frac{ \theta^{ B_k } }{ \Esp{ \theta^{ B_k } } } \bullet \Pp_{ B_k }
\end{align*}

Let $ \gamma_n := H^{(\Pe)}_n $ and
\begin{align*}
C'_n := \sum_{\ell = 1}^{k_n} B'_\ell(1/\gamma_n)
\end{align*}

Let $ (I_\ell)_\ell $ be a sequence of i.i.d. random variables in $ \intcrochet{1, \pi(n)} $ independent of $ (B_k)_k $ and $ (B'_k)_k $ distributed according to 
\begin{align*}
\Prob{ I = k } =  \frac{  \frac{1}{ \ppee_k + v_n  +1 }  }{  \sum_{\ell = 1}^{ \pi(n) } \frac{1}{ \ppee_\ell + v_n  + 1 } }, \qquad k \in \intcrochet{1, \pi(n) } 
\end{align*}

Last, let $ \delta(I_1, \dots , I_k ) $ be the length of the random partition created by means of the paintbox process \eqref{Def:Paintbox} associated to $ (I_\ell)_\ell $.

Then, the random variable $ \Omega''_n $ defined by 
\begin{align*}
\Omega''_n   :=   \sum_{ k \neq I_1, \dots , I_{C'_n} } B_{\ppee_k}(v_n) + \delta(I_1, \dots, I_{C'_n}) 
\end{align*}
is such that
\begin{align*}
\prth{ \Omega''_n, H_n^{(\Pe)} } \cvmodp{n}{ + \infty } \Phi_\omega 
\end{align*}
and defines hence a more accurate model of $ \omega(U_n) $ than the usual model $  \Omega_n := \sum_{ k \leq \pi(n) } B_{\ppee_k} $. As a consequence of mod-Poisson convergence, one has in addition the CLT
\begin{align*}
\frac{\Omega''_n - \log\log n }{ \sqrt{\log \log n} } \cvlaw{n}{ + \infty } \Ns(0, 1) 
\end{align*}
\end{theorem}


\begin{proof}
The idea to construct such a random variable lies on two approximations : approximate the random variable $ P_{\gamma_n } \sim \Ps(\gamma_n)  $ with the independent model $ \Omega_n $ since at the first order of convergence (i.e. convergence in law) these random variables are close, and approximate the limiting function $ \Phi_\omega $ by a suitable truncation $ \Phi_\omega^{(n)} $ of its product since such a finite product converges locally uniformly to $ \Phi_\omega $.

$ $

\noindent\underline{\textbf{First step :}} Change of random variable : With $ \gamma_n := H_n^{(\Pe)} $, we have 
\begin{align*}
\Pp_{ \Omega'_n }  & :=  \frac{ \Phi_C\prth{ \frac{ \Omega_n }{\gamma_n } } }{  \Esp{ \Phi_C\prth{ \frac{ \Omega_n }{\gamma_n } }  } } \bullet \Pp_{ \Omega_n } \\
\Pp_{ \Omega_n(x) } & :=   \frac{ x^{ \Omega_n  }  }{  \Esp{ x^{ \Omega_n }  } } \bullet \Pp_{ \Omega_n } 
\end{align*}

Remark that 
\begin{align*}
\frac{ \Esp{ x^{ \Omega'_n } } }{ \Esp{ x^{ P_{\gamma_n} } } }    =    \frac{ \Esp{ \Phi_C\prth{ \frac{ \Omega_n }{\gamma_n } } x^{  \Omega_n } }  }{ \Esp{ \Phi_C\prth{ \frac{ \Omega_n }{\gamma_n } }  } \Esp{  x^{  \Omega_n } }  } \, \frac{ \Esp{ x^{ \Omega_n} } }{ \Esp{ x^{ P_{\gamma_n} } } }  =  \frac{ \Esp{ \Phi_C\prth{ \frac{ \Omega_n(x) }{\gamma_n } } }  }{ \Esp{ \Phi_C\prth{ \frac{ \Omega_n }{\gamma_n } }  } }  \, \frac{ \Esp{ x^{ \Omega_n} } }{ \Esp{ x^{ P_{\gamma_n} } } }
\end{align*}

By mod-Poisson convergence, we have, locally uniformly in $ x \in \Rr_+ $
\begin{align*}
\frac{ \Esp{ x^{ \Omega_n} } }{ \Esp{ x^{ P_{\gamma_n} } } } \tendvers{n}{ + \infty } \Phi_\Omega(x)
\end{align*}

By dominated convergence, continuity of $ \Phi_C $ and law of large numbers for $ \Omega_n $ (i.e. $ \Omega_n / \gamma_n \to~1 $ almost surely and in $ L^1 $), we have
\begin{align*}
\Esp{ \Phi_C\prth{ \frac{ \Omega_n }{\gamma_n } } } \tendvers{n}{ + \infty } \Phi_C(1) = 1
\end{align*}

Last, using lemma \ref{CouplageBiaisExp}, we see that 
\begin{align*}
\Omega_n(x) \eqlaw \sum_{ p \leq n, p \in \Pe } B_p(x) 
\end{align*}
with  $ \Prob{ B_p(x) = 1 } = \frac{x / p}{x /p + 1 - 1/p } = \frac{x }{ p + x - 1 } $. Using the law of large numbers and the fact that 
\begin{align*}
\frac{  \sum_{p \leq n}  \frac{x }{ p + x - 1 } }{ \sum_{p \leq n}  \frac{1 }{ p  } } \tendvers{n}{ + \infty } x 
\end{align*}
we get that $ \Omega_n(x) / \gamma_n \to x $ almost surely\footnote{We suppose that $ x $ and the Bernoulli random variables are defined on the same probability space.} and in $ L^1 $, which implies by continuity of $ \Phi_C $ and dominated convergence that
\begin{align*}
\Esp{ \Phi_C\prth{ \frac{ \Omega_n(x) }{\gamma_n } } } \tendvers{n}{ + \infty } \Phi_C(x) 
\end{align*}

Hence, we have a first random variable that converges in the mod-Poisson sense to $ \Phi_\omega = \Phi_C \Phi_\Omega :$ 
\begin{align*}
\prth{ \Omega'_n, H_n^{(\Pe)} } \cvmodp{n}{ + \infty } \Phi_\omega 
\end{align*}

$ $

\noindent\underline{\textbf{Second step :}} Truncation of $ \Phi_C $ : Let $ k \in \Nn^* $ and
\begin{align*}
\Phi_C^{ (k) }(x) := \prod_{\ell = 1}^k \prth{ 1 +  \frac{ x - 1 }{\ell} } e^{-  \frac{ x - 1 }{\ell} }
\end{align*}

We clearly have for all $ x \in \Rr $
\begin{align*}
\Phi_C^{ (k) }(x) \leq 1
\end{align*}
since every term of the product satisfies this inequality. Lemmas \ref{Lemma:OrderPhiC} and \ref{Lemma:OrderDeltaPhiC} show that when $ k \to + \infty $ 
\begin{align}\label{PhiEstimates} 
\norm{ \Phi_C^{ (k) } - \Phi_C  }_\infty \!\!  =  O\prth{ \frac{1}{ \sqrt{k} } } \quad\quad \mbox{and} \quad\quad \norm{ D \Phi_C^{ (k) } }_\infty \!\! = O\prth{ 1 }
\end{align}
where the supremum is taken for $ x \in \Rr_+ $ and where $ Df(x) := f'(x) $. Last,
\begin{align*}
\abs{ \Esp{ \Phi_C^{ (k) }\prth{ \frac{\Omega_n(x)}{\gamma_n } }  } - \Phi_C(x) }  \leq  \norm{ \Phi_C^{ (k) } - \Phi_C  }_\infty  + \norm{ D \Phi_C^{ (k) } }_\infty \Esp{ \abs{ \frac{\Omega_n(x)}{\gamma_n } - x } }
\end{align*}

The classical CLT for sums of independent random variables ensures that
\begin{align*}
\Esp{ \abs{ \frac{\Omega_n(x)}{\gamma_n } - x } } = O_x\prth{\frac{ 1 }{ \sqrt{\gamma_n } } }
\end{align*}

Thus, taking $ k = k_n = \pe{\gamma_n} $, one has
\begin{align}\label{Ineq:EstimateTruncatedPhi}
\abs{ \Esp{ \Phi_C^{ (k) }\prth{ \frac{\Omega_n(x)}{\gamma_n } }  } - \Phi_C(x) } =  O_x \prth{\frac{ 1 }{ \sqrt{\gamma_n} } } 
\end{align}


We can now define
\begin{align*}
\Pp_{ \Omega''_n } & :=  \frac{ \Phi_C^{ (k_n) }\prth{ \frac{ \Omega_n }{\gamma_n } } }{  \Esp{ \Phi_C^{ (k_n) }\prth{ \frac{ \Omega_n }{\gamma_n } }  } } \bullet \Pp_{ \Omega_n }
\end{align*}
and this definition implies
\begin{align*}
\frac{ \Esp{ x^{ \Omega''_n } } }{ \Esp{ x^{ P_{\gamma_n} } } }  =  \frac{ \Esp{ \Phi_C^{ (k_n) }\prth{ \frac{ \Omega_n }{\gamma_n } } x^{ \Omega_n } } }{ \Esp{ \Phi_C^{ (k_n) }\prth{ \frac{ \Omega_n }{\gamma_n } }  } \Esp{ x^{ P_{\gamma_n} } } }  =     \frac{ \Esp{ \Phi_C^{ (k_n) } \! \prth{ \frac{ \Omega_n(x) }{\gamma_n } } }  }{ \Esp{ \Phi_C^{ (k_n) } \! \prth{ \frac{ \Omega_n(1) }{\gamma_n } }  } }  \, \frac{ \Esp{ x^{ \Omega_n} } }{ \Esp{ x^{ P_{\gamma_n} } } }
\end{align*}

Using \eqref{Ineq:EstimateTruncatedPhi}, this last quantity still converges locally uniformly to $ \Phi_\omega $, i.e. 
\begin{align*}
\prth{ \Omega''_n, H_n^{(\Pe)} } \cvmodp{n}{ + \infty } \Phi_\omega 
\end{align*}

We now construct $ \Omega''_n $ pathwise by means of a sequence of Bernoulli and uniform random variables.

$ $

\noindent\underline{\textbf{Third step :}} Construction. Let $ (B'_\ell )_\ell $ a sequence of independent $ \ensemble{0, 1} $-Bernoulli random variables with $ \Prob{B'_\ell = 1 } = \frac{1}{ \ell } $ independent of $ \Omega_n $. We have
\begin{align*}
\Esp{ x^{ \Omega''_n } } & = \frac{ 1}{ c_n } \Esp{  x^{ \Omega_n } \prod_{ \ell = 1 }^{ k_n } \prth{ 1 + \frac{1}{\ell} \prth{ \frac{\Omega_n}{\gamma_n } - 1 } } e^{ - \frac{ \Omega_n }{ \ell \gamma_n  } }   }  \\
					& =  \frac{ 1 }{ c'_n } \Esp{  (x v_n)^{ \Omega_n } \prod_{ \ell = 1 }^{ k_n } \prth{  \frac{ \Omega_n }{ \gamma_n  } }^{B'_\ell }    }  \ \ \ \mbox{with  } \   v_n := \exp\prth{ - H_{k_n } / \gamma_n } \\
					& =  \frac{ 1 }{ c''_n } \Esp{  (x v_n)^{ \Omega_n } \prod_{ \ell = 1 }^{ k_n } \Omega_n^{B'_\ell(1/\gamma_n) }    } \ \ \ \mbox{with the notations of lemma \ref{CouplageBiaisExp} } 
\end{align*}

Setting 
\begin{align*}
C'_n := \sum_{\ell \leq k_n} B'_\ell(1/\gamma_n)
\end{align*}
we get
\begin{align*}
\Esp{ x^{ \Omega''_n } } = \frac{ 1 }{ c_n } \Esp{  (x v_n)^{ \Omega_n } \Omega_n^{ C'_n }    } =  \frac{\Esp{  v_n^{ \Omega_n } \Omega_n^{ C'_n }   x^{ \Omega_n }  }}{ \Esp{   v_n^{ \Omega_n } \Omega_n^{ C'_n }  } }
\end{align*}

This random variable is the combination of two biases : a first exponential bias in the vein of lemma \ref{CouplageBiaisExp} with parameter $ v_n $ and a random iteration of size-bias transform, the number of times this transform is applied being given by $ C'_n $. The effect of the exponential bias amounts to change the probabilities of $ \Omega_n $ to get 
\begin{align*}
\widetilde{\Omega}_n \eqlaw \sum_{k = 1}^{\pi(n)} B_{ \ppee_k }(v_n)
\end{align*}
with $ \pi(n) := \sum_{p \in \Pe} \Unens{p \leq n} \sim \frac{n}{\log n} $ by the Prime Number Theorem, and
\begin{align}\label{EqLawOmegas}
\Esp{ x^{ \Omega''_n } }  =   \frac{ \Esp{   \widetilde{\Omega}_n^{ C'_n }   x^{ \widetilde{\Omega}_n }  } }{ \Esp{   \widetilde{\Omega}_n^{ C'_n }  } } = \frac{ 1 }{ \Esp{   \widetilde{\Omega}_n^{ C'_n }  } } \sum_{ \ell = 0 }^{k_n } \Prob{ C'_n = \ell } \Esp{   \widetilde{\Omega}_n^{ \ell}   x^{ \widetilde{\Omega}_n }  }
\end{align}

A size bias with a power $ \ell $ is nothing else than the $ \ell $-th iteration of the usual size-bias transform defined in lemma \ref{CouplageBiaisTaille}, as one can see by writing, for a bounded measurable function~$ f $
\begin{align*}
\frac{\Esp{ X^2 f(X) }}{ \Esp{ X^2} } & =: \frac{ 1  }{ \Esp{ X^2} }  \Esp{ X g(X) } \ \ \mbox{with }  g(x) := xf(x) \\
							& = \frac{ \Esp{ X } }{ \Esp{ X^2} }  \Esp{  g\prth{ X^{(s)} } } = \frac{ \Esp{ X } }{ \Esp{ X^2} } \Esp{  X^{(s)} f\prth{ X^{(s)} } } \\
							& =  \frac{ \Esp{ X }  }{ \Esp{ X^2} } \Esp{ X^{(s)} } \Esp{   f\prth{ (X^{(s)} )^{(s)} } }
\end{align*}
and one can check setting $ f = id : x \mapsto x $ in the definition of the size-bias transform that 
\begin{align*}
\Esp{ X^{(s)} } = \frac{ \Esp{ X^2 }  }{ \Esp{ X } }
\end{align*}
i.e. 
\begin{align*}
\Pp_{ X^{(s, 2)} } := \Pp_{ (X^{(s)} )^{(s)} } = \frac{ X^2 }{ \Esp{X^2 } } \bullet \Pp_X 
\end{align*}

From now on, we denote by $ X^{(s, k)} := (X^{(s, k-1)})^{(s )} $ and $ X^{(s, 0)} := X $. In virtue of lemma \ref{CouplageBiaisTaille}, we have
\begin{align*}
\widetilde{\Omega}_n^{(s )} \eqlaw \widetilde{\Omega}_n - B_{\ppee_I}(v_n) +  B_{\ppee_I}(v_n)^{(s )}
\end{align*}
with $ I \in  \intcrochet{1, \pi(n) }  $ a random index independent of all random variables in presence of distribution
\begin{align}\label{LoiDeI}
\Prob{ I = k } =  \frac{  \frac{1}{ \ppee_k + v_n  +1 }  }{  \sum_{\ell = 1}^{ \pi(n) } \frac{1}{ \ppee_\ell + v_n  + 1 } }
\end{align}

In addition, for a $ \ensemble{0, 1}$-Bernoulli random variable $ B $, we have
\begin{align*}
\Esp{ x^{ B^{(s )} } } = \frac{\Esp{ B x^{ B } }}{\Esp{ B  }} = x 
\end{align*}
i.e. $ B^{(s )} = 1 $ almost surely (which amounts to change its parameter to $1$). Hence, 
\begin{align*}
\widetilde{\Omega}_n^{(s )} \eqlaw \widetilde{\Omega}_n - B_{\ppee_I}(v_n) +  1 = \sum_{ k \leq \pi(n), \, k \neq I } B_{\ppee_k}(v_n) + 1
\end{align*}

If we iterate the transformation, this amounts to select at random a certain variable $ J $ whose law is given by \eqref{LoiDeI} independent of all random variables in presence and in particular independent of $ I $. Two cases can occur : either $ I = J $ in which case  $ \widetilde{\Omega}_n^{(s,2 )} \eqlaw \widetilde{\Omega}_n^{(s )} $, or  $ I \neq J $ in which case $ \widetilde{\Omega}_n^{(s,2 )} \eqlaw  \sum_{  k \neq I, J } B_{\ppee_k}(v_n) + 2 $, which we can sumarize into
\begin{align*}
\widetilde{\Omega}_n^{(s, 2 )} \eqlaw \sum_{ k \neq I, J } B_{\ppee_k}(v_n) + 1 + \Unens{ I \neq J }
\end{align*}

The third iterate gives 
\begin{align*}
\widetilde{\Omega}_n^{(s, 3 )} \eqlaw \sum_{ k \neq I_1, I_2, I_3 } B_{\ppee_k}(v_n) + \delta(I_1, I_2, I_3)
\end{align*}
with 
\begin{align*}
\delta(I_1, I_2, I_3) = \begin{cases}  1 \ \mbox{  if }  I_1 = I_2 = I_3 \\ 
2 \ \mbox{  if }  I_i = I_j \neq I_k \ \mbox{ for } \ensemble{i, j, k} = \ensemble{1, 2, 3} \\
3 \ \mbox{  if }  I_1 \neq I_2 \neq I_3 \neq I_1 \end{cases}
\end{align*}

At the $ \ell $-th iteration, one has, with a sequence of i.i.d. indexes $ (I_\ell)_\ell $ of law given by \eqref{LoiDeI}
\begin{align*}
\widetilde{\Omega}_n^{(s, \ell )} \eqlaw \sum_{ k \neq I_1, \dots , I_\ell } B_{\ppee_k}(v_n) + \delta(I_1, \dots, I_\ell)
\end{align*}
where $ \delta(I_1, \dots, I_\ell) $ is the length of the random partition $ \lambda \vdash \ell $ constructed by means of the paintbox process \eqref{Def:Paintbox} associated with the i.i.d. sequence $ (I_1, \dots, I_\ell) $ (see \cite{PitmanStFlour}). In the case where $ I \sim \Us\prth{ \intcrochet{1, \pi(n) } } $, this random partition is equal in law to the cycle structure of a random uniform permutation $ \sigma \in  \Sg_{\ell} $ and in particular, $ \delta(I_1, \dots, I_\ell) = C(\sigma) $. In the case of our indexes, this distribution has still to be precised.

Last, the equality \eqref{EqLawOmegas} is equivalent to 
\begin{align*}
\Omega''_n   \eqlaw   \widetilde{\Omega}_n^{(s, C'_n )} 
\end{align*}
which implies that 
\begin{align*}
\Omega''_n   \eqlaw   \sum_{ k \neq I_1, \dots , I_{C'_n} } B_{\ppee_k}(v_n) + \delta(I_1, \dots, I_{C'_n}) \ \ \mbox{with } \ \ C'_n := \sum_{\ell \leq k_n} B'_\ell(1/\gamma_n)
\end{align*}
all the random variables considered being independent.
\end{proof}


\begin{remark} Note the following rewriting of the corrective term : one has refined the Erdös-Kac model $ \Omega_n $ by imposing a certain proportion of primes (in quantity $ \delta(I_1, \dots, I_{C'_n}) $) to be divisors with probability one. Knowing that certain primes are divisors allows to avoid them in the set of primes to consider to define the model. This operation is thus a conditioning :
\begin{align}\label{InterpConditionnement}
\Omega''_n \eqlaw \prth{ \Omega_n(v_n) \Big\vert B_{ \ppee_{I_1} }(v_n) = 1, \dots,  B_{ \ppee_{I_{C'_n	} } }(v_n) = 1}
\end{align}

\end{remark}


\begin{remark}
The result here proven shows the apparition of an approximate coalescing structure of the $ B_p^{(n)} $'s (the paintbox partition) in place of a more classical distribution on permutations. Such a ``coalescence of primes'' (or, at least, the $ B_p^{(n)} $) seems to be new. It is a natural consequence of iterated size-biasing of a sum of random variables (or a random walk) as proven in theorem \ref{Theorem:ErdosKacImprovement}, and the result can be considered in itself in the framework of pure probabilistic terms. Note that this last fact is a natural direct consequence of the bias interpretation of mod-Poisson convergence and of the form of the limiting function when Bernoulli random variables are involved. 
\end{remark}


\begin{remark} 

A drawback of this model is the fact that the primes randomly selected are not the large or small ones ; they are selected at random in the whole interval $ \intcrochet{1, \pi(n) } $ and not above or below a certain treshold ; this is the type of explanation that one would like to evoke for the appearance of the corrective term (the appearance of lots of particular primes in the prime decomposition).
Nevertheless, due to the structure of the law \eqref{LoiDeI}, there is a strong probability to only select the small primes in the conditioning.

Last, the change of probability can also be thought of as another drawback for one may prefer to have an identity of the type
\begin{align*}
\Omega''_n  \eqlaw  \prth{ \Omega_n \Big\vert B_{ \ppee_{I_1} } = 1, \dots,  B_{ \ppee_{I_{C'_n	} } } = 1}
\end{align*}
and get a refined model by means of the sole conditioning. Nevertheless, as this change of the parameters of the Bernoulli random variables is a bias, this is in accordance with Hardy and Littlewood's modification of the Cramer model by multiplying its probabilities with a suitable factor (namely, by biasing). Here, as in lots of other models, the natural modification of a coarse model consists in biasing it with a general weight and not only with indicators of particular events. 
\end{remark}

\begin{remark}
The choice of parameters used in theorem \ref{Theorem:ErdosKacImprovement} chosen to match the limiting generating functions is similar to the modification of the Cramer model to incorporate the twin primes, see \cite{TaoBlogCramer}.
\end{remark}

\section{Conclusion and perspectives}

The study of $ \omega(U_n) $ is fundamental for the understanding of the repartition of the primes, and to this goal, it would be interesting to go beyond the Erdös-Kac theorem and to look for instance at Beurling primes, where a Sath\'e-Selberg theorem was proven in \cite{Rupert} or at a functional renormalisation, i.e. the Erdös-Kubilius theorem (see e.g. \cite{Tenenbaum2}). Such a functional renormalisation involves a Brownian motion at the limit, but a more refined one would give a Poisson process in the same manner the Poisson distribution appears for $ \omega(U_n) $. Note that the functional generalisation of mod-Poisson convergence in the functional setting by means of a functional Fourier or Laplace transform is straightforward.

More generally, a better construction of the mod-Poisson model for $ \omega(U_n) $ has to be done. A general guess would be a random variable of the type 
\begin{align*}
\sum_{ k \neq I_1, \dots , I_{Z_n} } B_{\ppee_k} + \delta(I_1, \dots, I_{Z_n}) 
\end{align*}
with the $ (I_\ell)_\ell $ independent uniform random variables on $ \intcrochet{ A, \pi(n) } $ where $ A $ is to be found, and $ Z_n $ a random integer to be found. The advantages of such a model is the natural apparition of the number of cycles with the paintbox process (since, with the $ I_\ell $'s uniform, this is the number of cycles under the Haar measure of a certain symmetric group), and the interpretation in terms of conditioning on the large primes. But such a hypothetical model relies heavily on the nature of the random variable $ Z_n $, whose distribution is yet to be discovered.

Another generalisation of this result consists in transposing it into the general framework of a product of two $ \ensemble{0, 1} $-Bernoulli mod-Poisson limiting functions given in \eqref{LimitePoissonGenerale}. In particular, let  $ q = p^\nu $ for $ p \in \Pe $ and $ \nu \in \Nn^* $ and let $ \Ff_q $ denote the field with $q$ elements. Denote by  $ \Pe(\Ff_q[X]) $ the irreducible monic polynomials of $ \Ff_q[X] $ and by $ \omega_q(P_n) $ the number of divisors of $ P_n $ defined, for all monic $ Q \in \Pe(\Ff_q[X]) $ by 
\begin{align*}
\omega_q(Q) := \sum_{ \pi \in \Pe(\Ff_q[X]) } \Unens{ \pi \ddivise Q }
\end{align*}

Let $ Q_n $ be a random monic polynomial of degree equal to $n$ selected according to the uniform measure of this finite set. It is shown in \cite{KowalskiNikeghbali2} that
\begin{align*}
\frac{ \Esp{ z^{\omega_q(Q_n) } } }{ \Esp{ z^{ P_{\gamma_n } } } } = \Phi_{\omega_q}(z) (1 + o(1) ) 
\end{align*}
where $ \gamma_n = H_n  + O(1) $, where $ \abs{ \pi }_q  := q^{ \deg(\pi) } $, and where
\begin{align*}
\Phi_{\omega_q}(z) = \frac{e^{ -(z-1) \gamma }}{ \Gamma(z) } \prod_{ \pi \in \Pe(\Ff_q[X]) } \prth{ 1 + \frac{z - 1}{ \abs{ \pi }_q } } e^{ - \frac{z - 1}{\abs{ \pi }_q} } =: \Phi_C(z) \Phi_{\Omega_q}(z)
\end{align*}

This form is reminiscent of \eqref{Thm:SplittingModPoisson}, with a corrective model given by $ C(\sigma_n) $ for $ \sigma_n $ a random uniform permutation of $ \Sg_n $, and an independent model given by
\begin{align*}
\Omega_{q, n} := \sum_{ \underset{\pi \in \Pe(\Ff_q[X])}{ \deg(\pi) \leq n } } B_\pi = \sum_{ d = 1 }^n \sum_{ \underset{\pi \in \Pe(\Ff_q[X])}{ \deg(\pi) = d } } B_\pi
\end{align*}
where the $ (B_\pi)_\pi $ are independent Bernoulli random variables such that
\begin{align*}
\Prob{B_\pi = 1} = \frac{1}{\abs{\pi}_q } = 1 -  \Prob{B_\pi = 0}
\end{align*}

Here, the generalisation of theorem \ref{Theorem:ErdosKacImprovement} is straightforward, but if the rewriting in terms of bias-conditioning holds, the model without bias is still to be constructed.

\begin{remark}
Note another decomposition of $ \omega_q(Q_n) $ given in \cite[Prop 9.7]{RodgersSymArithm} that characterises its oscillation around $ \lim_{q \to +\infty}\Esp{\omega_q(Q_n) } = H_n $. This elegant formula uses directly the representation theory of the symmetric group and is in the vein of \cite{KowalskiNikeghbali2} that study this random variable by decomposing it in two parts, a squarefree part and a remainder. 
\end{remark}

Last, one could also try to adapt theorem \ref{Theorem:ErdosKacImprovement} to different functionals than $ \omega $, for instance the total number of prime divisors of $ U_n $ (if $ U_n = \prod_{p \in \Pe} p^{v_p(U_n) } $, it is defined as $ \sum_p v_p(U_n) $ ; this is the other random variable that satisfies a mod-Poisson convergence in the Sath\'e-Selberg theorems) or to $ \log\abs{ \zeta(1/2 + i t U } $, which was the initial motivation.

\section*{Acknowledgements}

The author expresses his thanks to O. H\'enard, A. Nikeghbali, J. Najnudel and N. Zygouras for several remarks concerning earlier drafts of the paper. The first version of this paper was written while the author was a guest at the University of California Irvine ; many thanks are due to this institution and in particular to M. Cranston.

The author was supported by EPSRC grant EP/L012154/1 and the Schweizerischer Nationalfonds PDFMP2 134897/1.


\bibliographystyle{amsplain}

\begin{thebibliography}{10}


\bibitem{ArratiaBarbourTavare}
\rm R. Arratia, A. D. Barbour, and S. Tavaré, 
\it Logarithmic combinatorial structures, a probabilistic approach, 
\rm EMS Monographs in Mathematics, Z\"urich, Europ. Math. Soc. (\textbf{2003}).



\bibitem{BarhoumiModStein}
\rm Y. Barhoumi-Andr\'eani,
\it On Stein's method and mod-* convergence,
\rm preprint \url{https://arxiv.org/pdf/1701.03086v1.pdf} (\textbf{2017}). 


\bibitem{CramerRandom}
\rm H. Cram\'er, 
\it On the order of magnitude of the difference between consecutive prime numbers, 
\rm Acta Arithmetica 2:23-46 (\textbf{1936}).



\bibitem{ErdosKac}
\rm P. Erd\"os, M. Kac,
\it The Gaussian law of errors in the theory of additive functions,
\rm Proc. N. A. S., vol. 25, pp. 206--207 (\textbf{1939}). 


\bibitem{ErdosKac2}
\rm P. Erd\"os, M. Kac,
\it On the Gaussian law of errors in the theory of additive number theoretic functions,
\rm Amer. J. Math., vol. 62, pp. 738--742 (\textbf{1940}). 



\bibitem{FerayMeliotNikeghbali}
\rm V. Féray, P.-L. Meliot, A. Nikeghbali,
\it Mod-phi convergence I: Normality zones and precise deviations,
\rm preprint \url{https://arxiv.org/pdf/1304.2934v4.pdf} (\textbf{2013}).



\bibitem{GonekHughesKeating}
\rm M. Gonek, C. P. Hughes, J. Keating,
\it A hybrid Euler-Hadamard formula for the Riemann Zeta function, 
\rm Duke Math. J., vol. 136, Number 3, pp. 507--549 (\textbf{2007}).


\bibitem{HardyLittlewoodOnCramer}
\rm G. H. Hardy, J. E. Littlewood,
\it Some problems of Parititio Numerorum (III): On the expression of a number as a sum of primes, 
\rm Acta Math., 44:1-70 (\textbf{1922}).


\bibitem{JacodAl}
\rm J. Jacod, E. Kowalski, A. Nikeghbali,
\it Mod-Gaussian convergence: new limit theorems in probability and number theory,
\rm Forum Mathematicum, vol. 23 (4), pp. 835--873 (\textbf{2011}).



\bibitem{Joyner}
\rm D. Joyner, 
\it Distribution theorems of $L$-functions,
\rm Pitman Res. Notes in Math. Series, vol. 142. Longman Sc. \& Tech., Harlow ; John Wiley \& Sons, Inc., New York, (\textbf{1986}).



\bibitem{KeatingSnaith}
\rm J.P. Keating, N.C. Snaith,
\it Random Matrix Theory and $ \zeta(1/2 + i t)$,
\rm Comm. Math. Phys., vol. 214, pp. 57--89 (\textbf{2000}).


\bibitem{KowalskiNikeghbali2}
\rm E. Kowalski and A. Nikeghbali,
\it Mod-Poisson convergence in probability and number theory,
\rm Intern. Math. Res. Not., vol. 18, pp. 3549--3587 (\textbf{2010}).


\bibitem{LemkeSound}
\rm R. J. Lemke Oliver, K. Soundararajan,
\it Unexpected biases in the distribution of consecutive primes, 
\rm \url{http://arxiv.org/abs/1603.03720} (\textbf{2016}).


\bibitem{PitmanStFlour}
\rm J. Pitman,
\it Combinatorial stochastic processes,
\rm St-Flour summer school XXXII-2002, Springer, New York (\textbf{2008}).


\bibitem{RodgersSymArithm}
\rm B. Rodgers,
\it Arithmetic functions in short intervals and the symmetric group, 
\rm \url{https://arxiv.org/pdf/1609.02967.pdf} (\textbf{2016}). 


\bibitem{Rupert}
\rm M. Rupert,
\it Extending Erd\"os-Kac and Selberg-Sathe to Beurling primes with controlled integer counting functions, 
\rm University of British Columbia MSc thesis (\textbf{2013}). 


\bibitem{SelbergSathe}
\rm A. Selberg,
\it Note on a paper by L.G. Sathe,
\rm J. Indian Math. Soc., vol. 18 pp. 83--87 (\textbf{1954}). 


\bibitem{Sathe}
\rm L.G. Sathe,
\it On a problem of Hardy on the distribution of integers having a given number of prime factors,
\rm J. Indian Math. Soc., vol. 17 pp. 63--141 (\textbf{1953}). 


\bibitem{SoundSmallGaps}
\rm K. Soundararajan,
\it Small gaps between prime numbers: the work of Goldston-Pintz-Yildrim, 
\rm Bull. (new ser.) of the AMS, 44(1):1-18 (\textbf{2007}).


\bibitem{TaoBlogCramer}
\rm T. Tao,
\it Probabilistic models and heuristics for the primes,
\rm available at \url{https://terrytao.wordpress.com/2015/01/04/254a-supplement-4-probabilistic-models-and-heuristics-for-the-primes-optional} (\textbf{2015}).


\bibitem{Tenenbaum}
\rm G. Tenenbaum,
\it Introduction to analytic and probabilistic number theory,
\rm Cambridge Studies in Advanced Mathematics, Cambridge University Press (\textbf{1995}). 


\bibitem{Tenenbaum2}
\rm G. Tenenbaum,
\it Qu'est-ce qu'un entier normal ?
\rm in Leçons de mathématiques d'aujourd'hui, vol. 2, available at \url{http://www.iecn.u-nancy.fr/~tenenb/PUBLIC/PPP/LMA.pdf} (\textbf{1995}).


\bibitem{WhittakerWatson}
\rm E.T. Whittaker, G.N. Watson,
\it A course in modern analysis,
\rm 4th Edition, Cambridge Math. Library, Cambridge University Press (\textbf{1996}). 


\end{thebibliography}


\end{document}